\documentclass[twoside,11pt]{article}

\usepackage{graphicx, subfigure}
\usepackage[top=1in,bottom=1in,left=1in,right=1in]{geometry}
\usepackage[sort&compress, numbers]{natbib} 
\usepackage{amsfonts, amsmath, amssymb, amsthm, constants, bbm}
\usepackage{mathtools}
\mathtoolsset{showonlyrefs}
\usepackage{enumitem}

\usepackage{color}
\definecolor{darkred}{RGB}{100,0,0}
\definecolor{darkgreen}{RGB}{0,100,0}
\definecolor{darkblue}{RGB}{0,0,150}

\usepackage{hyperref}
\hypersetup{colorlinks=true, linkcolor=darkred, citecolor=darkgreen, urlcolor=darkblue}
\usepackage{url}

\def\hc{{\rm HC}}

\newtheorem{thm}{Theorem}
\newtheorem{prp}{Proposition}

\theoremstyle{remark}

\def\beq{\begin{equation}} 
\def\eeq{\end{equation}}
\def\beqn{\begin{eqnarray*}}
\def\eeqn{\end{eqnarray*}}
\def\Bitem{\begin{itemize}\setlength{\itemsep}{.2in}}
\def\bitem{\begin{itemize}\setlength{\itemsep}{.05in}}
\def\eitem{\end{itemize}}
\def\Benum{\begin{enumerate}\setlength{\itemsep}{.2in}}
\def\benum{\begin{enumerate}\setlength{\itemsep}{.05in}}
\def\eenum{\end{enumerate}}
\def\bmult{\begin{multline*}}
\def\emult{\end{multline*}}
\def\bcenter{\begin{center}}
\def\ecenter{\end{center}}
\def\bframe{\begin{frame}}
\def\eframe{\end{frame}}

\newcommand{\figref}[1]{Figure~\ref{fig:#1}}

\def\cH{\mathcal{H}}

\def\cN{\mathcal{N}}

\def\bbI{\mathbb{I}}

\newcommand{\E}{\operatorname{\mathbb{E}}}
\renewcommand{\P}{\operatorname{\mathbb{P}}}
\newcommand{\Var}{\operatorname{Var}}

\def\Bin{\text{Bin}}

\def\eps{\varepsilon}

\def\1{\mathbbm{1}}
\newcommand{\IND}[1]{\bbI\{ #1 \}}

\definecolor{purple}{rgb}{0.4,.1,.9}

\newcommand\blfootnote[1]{%
  \begingroup
  \renewcommand\thefootnote{}\footnote{#1}%
  \addtocounter{footnote}{-1}%
  \endgroup
}

\begin{document}
\thispagestyle{empty}

\title{Detecting Sparse Heterogeneous Mixtures in a Two-Sample Problem}
\author{Rong Huang}
\date{}
\maketitle

\begin{abstract}
    We consider the problem of detecting sparse heterogeneous mixtures in a two-sample setting from a nonparametric perspective, where the effect manifests itself as a positive shift. We suggest a two-sample higher criticism test, and show that it is first-order comparable to the likelihood ratio test for the generalized Gaussian mixture models in all sparsity regimes. 
\end{abstract}

\blfootnote{The author is with the Department of Mathematics, University of California, San Diego (\url{math.ucsd.edu}).}

\section{Introduction}
The detection of sparse mixtures has been studied for decades \cite{ingster1997some,Jin:2004fj}. Most work has focused on detecting deviation of data from the known distribution. However, in practice, it's more common that we do not have access to the null distribution and have to estimate it from a control group. For example, in a clinical trial, patients were assigned to two groups randomly, given either the placebo or the treatment. For some reasons, the treatment could only affect a small proportion of the patients treated, while the remaining patients reacted the same as patients in the control group. Similar settings were investigated in Conover and Salsburg \cite{conover1988locally} and they modeled the shift in the distribution as a Lehmann alternative and focused on the locally most powerful tests.

We study this situation in parallel with the work of Ingster \cite{ingster1997some} and Donoho and Jin \cite{Jin:2004fj}. Let $F$ and $G$ be two continuous (unknown) distribution functions on the real line. We consider the following hypothesis testing problem: based on a random sample $X_1, \cdots, X_m$ drawn iid from $F$ and another independent random sample $Y_1, \cdots, Y_n$ drawn iid from $G$, decide
\beq\label{problem2sample}
\cH_0: G = F \quad \text{versus} \quad \cH_1: G = (1-\eps)F + \eps F(\cdot - \mu), \quad \eps > 0, \quad \mu > 0.
\eeq
where $\eps \in (0,1/2)$ is the fraction of non-null effect and $\mu$ is the size of the location shift. Hence, under the alternative, $G$ is stochastic larger than $F$. We assume that 
\beq \label{mn}
\lim_{m , n \to \infty} \frac{n}{m+n} = \eta \in (0,1/2]
\eeq 
at a sufficient fast rate.

The optimum of rank tests was mainly investigated as locally most powerful \cite{lehmann1953power}. Following the line of our work in the one-sample setting, we still focus on the asymptotically optimum here. As usual, a testing procedure is asymptotically powerful (resp. powerless) if the sum of its probabilities of Type I and Type II errors (its risk) has limit 0 (resp. inferior at least 1) in the large sample asymptote.

\subsection{A benchmark: generalized Gaussian mixture model}
The normal mixture model has been studied in Ingster \cite{ingster1997some} considering the one-sample setting, that is, $F$ is known to be standard normal and only the $Y$-sample is collected. The problem is investigated in various asymptotic regimes defined by how fast $\eps$ goes to zero. The detection boundary of the likelihood ratio test (LRT) (then any other tests) is derived.  Donoho and Jin \cite{Jin:2004fj} further derived the detection boundary when $F$ is generalized Gaussian: 
\beq
f(x) \propto \exp \Big(-\frac{|x|^\gamma}{\gamma}\Big),
\eeq 
where $\gamma > 0$. Note that $\gamma = 2$ corresponds to the normal distribution and $\gamma = 1$ corresponds to the double-exponential distribution.  They parameterized $\eps = \eps_n$ as 
\beq \label{eps2sample}
\eps_n = n^{-\beta}, \quad 0 < \beta < 1 \text{ fixed}.
\eeq
In the sparse setting where $1/2 < \beta < 1$, let 
\beq \label{mu}
\mu_n = (\gamma r \log n)^{1/\gamma}, \quad 0<r<1 \text{ fixed},
\eeq 
then the detection boundary when $\gamma > 1$ is
\beq 
\rho^*_\gamma (\beta) = \begin{cases} (2^{1/(\gamma-1)}-1)^{\gamma-1}(\beta - \frac{1}{2}), & \frac{1}{2} < \beta < 1 - 2 ^{-\gamma/(\gamma -1 )}; \\
(1-(1-\beta)^{1/\gamma})^\gamma, & 1 - 2 ^{-\gamma/(\gamma -1 )}< \beta <1.
\end{cases}
\eeq 
and for the case $\gamma \le 1$
\beq
\rho^*_\gamma (\beta) = 2 \beta - 1.
\eeq 
That means, if $r > \rho^*(\beta)$, $\cH_0$ and $\cH_1$ separate asymptotically, while if  $r < \rho^*(\beta)$, $\cH_0$ and $\cH_1$ merge asymptotically. 

The detection boundary in the dense regime where $0 < \beta < 1/2$ is given in \cite{AriasCastro:2016is}. Let
\beq \label{mu2}
\mu_n = n^{s-1/2}, \quad 0 < s < 1/2 \text{ fixed},
\eeq
then the hypotheses merge asymptotically when $s < \beta$ if $\gamma \ge 1/2$ and $s < \frac{1}{2} - \frac{1-2\beta}{1+ 2\gamma}$ if $\gamma < 1/2$.

\subsection{The two-sample higher criticism test}
The one-sample higher criticism test was suggested in Donoho and Jin \cite{Jin:2004fj} and proved to be first-order asymptotically comparable to the LRT in the normal mixture model. In the two-sample setting, an analogous higher criticism statistic was proposed as \cite{canner1975simulation, finner2018two}:
\beq \label{HC2sample}
\hc = \sup_{t \in \mathbb{R}} \sqrt{\frac{mn}{m+n}} \frac{[F_m(t) - G_n(t)]}{\sqrt{H_{m+n}(t)(1-H_{m+n}(t))}},
\eeq
where $F_m$ and $G_n$ are empirical distributions of the $X$-sample and the $Y$-sample, respectively, and $H_{m+n}(t) = \frac{1}{m+n} (m F_m(t) + n G_n(t))$ is the empirical distribution of the combined sample. The test rejects for large values of \eqref{HC2sample}. This is to the two-sample Kolmogorov - Smirnov test \cite{smirnov1939estimation} what the Anderson-Darling test (aka the higher criticism test) is to the Kolmogorov-Smirnov test. 

Pettitt \cite{pettitt1976two} proposed the integral version of the two-sample Anderson-Darling statistic and gave an approximation of the distribution. Finner and Gontscharuk \cite{finner2018two} studied the supremum version \eqref{HC2sample} in terms of local levels and focused on the Type I error.  It is also connected to the work of Zhao et al \cite{Zhao:9999ku}, which normalizes the Hoeffding test for independence in an analogous way.
Note that all these tests are based on ranks only, so they are nonparametric tests. Distribution-free tests for sparse heterogeneous mixtures in the one-sample setting was investigated in \cite{AriasCastro:2016is} where they assumed that $F$ is symmetric about zero and the true effects have positive median. In our work, we do not pose any assumptions on $F$ except it is continuous, and we use the $X$-sample to estimate $F$. In addition, the two-sample higher criticism is parallel to the CUSUM sign test in \cite{AriasCastro:2016is} as the Wilcoxon test to the Wilcoxon signed-rank test \cite{wilcoxon1945}. 

\section{Lower bound}
The formulation \eqref{problem2sample} indicates that both the null and the alternative hypotheses are composite. If we assume model parameters $(F, \eps, \mu)$ are known, then the likelihood ratio test (LRT) is the most powerful test by Neyman-Pearson lemma. In particular, if $F$ is given, we don't need the $X$-sample, then the question is reduced to the one-sample situation \cite{ingster1997some, Jin:2004fj, cai2014optimal}. The general detection boundary was given in \cite[Lemma A.1]{arias2018dist} as follows in our context: let $f$ denote the density of $F$, then the hypotheses \eqref{problem2sample} merge asymptotically when there is a sequence $(x_n)$ such that
\beq
n \bar  F (x_n) \to 0, \quad n \eps_n \bar F(x_n - \mu_n) \to 0,
\eeq
and
\beq
n \eps^2_n \Big [ \int_{-\infty}^{x_n} \frac{f(x-\mu_n)^2}{f(x)} dx - 1   \Big]_+ \to 0.
\eeq 

\section{The two-sample higher criticism test}
For a distribution $F$, $\bar F(x) = 1 - F(x)$ will denote its survival function. 
\begin{thm}
For the testing problem \eqref{problem2sample} and under \eqref{mn}, the two-sample higher criticism test is asymptotically powerful if either there is a sequence $(t_n)$ such that $t_n \to \infty$, 
\beq \label{condHC1}
n(\bar F(t_n) \vee \eps \bar F(t_n-\mu)) \gg \log^2 n,
\eeq 
and 
\beq \label{condHC2}
\frac{ \sqrt{n} \eps [\bar F(t_n - \mu) - \bar F(t_n)]}{\sqrt{\bar F(t_n) + \eps \eta \bar F(t_n-\mu)}} \gg \log n;
\eeq
or $t$ is the median of $F$ and 
\beq \label{condHC3}
 \sqrt{n} \eps [\bar F(t - \mu) - \frac{1}{2}] \gg \log n.
\eeq 
\end{thm}

\begin{proof} 
Finner and Gontscharuk \cite{finner2018two} showed that the two-sample HC statistic \eqref{HC2sample} is almost surely equal to 
\beq 
\hc^* = \sup_{s \in I_{m,n}} \sqrt{\frac{m+n}{m+n-1}} \frac{[V_{m,s}-\E_0[V_{m,s}]]}{\sqrt{\Var_0[V_{m,s}]}},
\eeq 
where $V_{m,s}$ denotes the number of ranks related to the $X$-sample being not larger than $s$ and $I_{m,n} := \{1, \cdots, m+n-1\}$. They also showed that $\hc^*$ coincides asymptotically in distribution with the one-sample HC statistic with sample size $n$ under the null as we assume that $n \le m$, which was derived in \cite{jaeschke1979asymptotic}. Thus, we have 
\beq
\P_0\big(\hc \ge \sqrt{3 \log \log n}\big) \to 0.
\eeq
For simplicity, we consider the test with rejection region $\{\hc \ge \log n\}$.
Hence, the test is asymptotically powerful if, under the alternative, there is $t_n$ (or $t$) $\in \mathbb{R}$ such that
\beq
\P_1 \Big (\sqrt{\frac{mn}{m+n}} \frac{[F_m(t_n) - G_n(t_n)]}{\sqrt{H_{m+n}(t_n)(1-H_{m+n}(t_n))}} \ge \log n \Big) \to 1.
\eeq
Indeed, $mF_m(t)$ is binomial with parameters $m$ and $F(t)$, $nG_n(t)$ is binomial with parameters $n$ and $G(t) = (1-\eps)F(t) + \eps F(t - \mu)$, and 
\beq 
H_{m+n}(t) = \frac{m}{m+n}F_m(t) + \frac{n}{m+n} G_n(t).
\eeq
We also define $H(\cdot)$ as
\beq 
H(t) =  (1-\eta)F(t) + \eta G(t) =  (1- \eps \eta)F(t) + \eps \eta F(t-\mu).
\eeq
Hence, by Chebyshev's inequality, we have
\beq 
\frac{\sqrt{m} |F_m(t_n) - F(t_n)|}{\sqrt{F(t_n)(1-F(t_n))}} \le \log m
\eeq
with probability tending to 1, and 
\beq 
\frac{\sqrt{n} |G_n(t_n) - G(t_n)|}{\sqrt{G(t_n)(1-G(t_n))}} \le \log n
\eeq
with probability tending to 1. By triangular inequality, we have
\begin{align}
    |F(t_n) - G(t_n)| &= |F(t_n) - F_m(t_n) + F_m(t_n)- G_n(t_n) +G_n(t_n) - G(t_n)|\\
    &\le |F(t_n) - F_m(t_n)| + |F_m(t_n)- G_n(t_n)| + |G_n(t_n) -  G(t_n)| \\
    & \le |F_m(t_n)- G_n(t_n)| + \log m \sqrt{\frac{F(t_n)(1-F(t_n))}{m}} + \log n \sqrt{\frac{G(t_n)(1-G(t_n))}{n}}
\end{align}
Hence, we have
\begin{align}
    \sqrt{\frac{mn}{m+n}} \frac{[F_m(t_n) - G_n(t_n)]}{\sqrt{H_{m+n}(t_n)(1-H_{m+n}(t_n))}} \ge a_n - b_n,
\end{align}
where
\begin{align}
a_n &:= \sqrt{\frac{mn}{m+n}} \frac{ \eps [\bar F(t_n - \mu) - \bar F(t_n)]}{\sqrt{H_{m+n}(t_n)(1-H_{m+n}(t_n))}}, \\
b_n &:= \sqrt{\frac{n}{m+n}} \log m \sqrt{\frac{F(t_n)(1-F(t_n))}{H_{m+n}(t_n)(1-H_{m+n}(t_n))}} + \sqrt{\frac{m}{m+n}} \log n \sqrt{\frac{G(t_n)(1-G(t_n))}{H_{m+n}(t_n)(1-H_{m+n}(t_n))}},
\end{align}
as we know that $F>G$ under the alternative, and it suffices to show that 
\beq \label{condHC}
a_n \ge b_n + \log n
\eeq
with probability tending to 1.

\begin{align}
    |\bar H_{m+n}(t_n) - \bar H(t_n)| &= \big |\frac{m}{m+n} \bar F_m(t_n) + \frac{n}{m+n} \bar G_n(t_n) - (1-\eta) \bar F(t_n) - \eta \bar G(t_n) \big| \\
    &= \big |\frac{m}{m+n}(\bar F_m(t_n)-\bar F(t_n)) + (\frac{m}{m+n}- (1-\eta))\bar F(t_n) \nonumber \\
    & \quad + \frac{n}{m+n} (\bar G_n(t_n)-\bar G(t_n))  + (\frac{n}{m+n} - \eta )\bar G(t_n) \big|\\
    & \le \big |\frac{m}{m+n}(\bar F_m(t_n)- \bar F(t_n))| + |\frac{n}{m+n} (\bar G_n(t_n)-\bar G(t_n))\big | \nonumber \\
    & \quad + O \Big (\frac{\log n}{\sqrt{n}} \Big) (\bar F(t_n) + \bar G(t_n)) \\
    & \le \frac{m}{m+n}\log m \sqrt{\frac{\bar F(t_n)(1- \bar F(t_n))}{m}} + \frac{n}{m+n} \log n \sqrt{\frac{\bar G(t_n)(1- \bar G(t_n))}{n}} \nonumber\\
    & \quad + O \Big (\frac{\log n}{\sqrt{n}} \Big) (\bar F(t_n) + \bar G(t_n)) \\
    & \asymp O \Big(  \frac{\log m}{\sqrt{m}} \sqrt{\bar F(t_n)}  \Big) +  O \Big(\frac{\log n}{\sqrt{n}} \sqrt{\bar G(t_n)}\Big) +  O \Big(  \frac{\log n}{\sqrt{n}}  (\bar F(t_n) + \bar G(t_n)) \Big) \\
    & \asymp O \Big ( \frac{\log n}{\sqrt{n}} \sqrt{(1- \eps)\bar F(t_n) + \eps \bar F(t_n-\mu)}  \Big),
\end{align}
as we assume that 
\beq 
|\frac{n}{m+n} - \eta| = O (\log n / \sqrt{n}).
\eeq 
Thus, under \eqref{condHC1} or \eqref{condHC3}, we have
\beq 
|\bar H_{m+n}(t_n) - \bar H(t_n)| \ll  \bar H(t_n).
\eeq 
Then if $t_n \to \infty$, we have $H(t_n) \to 1$, and
\begin{align}
    a_n &= \sqrt{\frac{mn}{m+n}} \frac{ \eps [ \bar F(t_n - \mu) - \bar F(t_n)]}{\sqrt{H_{m+n}(t_n)(1-H_{m+n}(t_n))}} \\
    & \asymp \frac{ \sqrt{n (1-\eta)} \eps [\bar F(t_n - \mu) - \bar F(t_n)]}{\sqrt{(1- \eps \eta)\bar F(t_n) + \eps \eta \bar F(t_n-\mu)}} \cdot \sqrt{\frac{\bar H(t_n)}{\bar H_{m+n}(t_n)}} \\
    & \asymp \frac{ \sqrt{n} \eps [\bar F(t_n - \mu) - \bar F(t_n)]}{\sqrt{\bar F(t_n) + \eps \eta \bar F(t_n-\mu)}} \gg \log n,
\end{align}
under \eqref{condHC2}. 

If $t$ is the median of $F$ such that $F(t) = 1/2$, then $H(t)$ is bounded away from 0 and 1, and under condition \eqref{condHC3},
\begin{align}
a_n &= \sqrt{\frac{mn}{m+n}} \frac{ \eps [\bar F(t - \mu) - \bar F(t)]}{\sqrt{H_{m+n}(t)(1-H_{m+n}(t))}} \\
& = \sqrt{n} \eps [\bar F(t - \mu) - \frac{1}{2}] \gg \log n.
\end{align}
In addition, we have
\beq 
b_n  \asymp \log n.
\eeq 
Therefore, \eqref{condHC} is fulfilled eventually.
\end{proof}

In the generalized Gaussian mixture model, with parameterization \eqref{eps2sample} and \eqref{mu}, we choose $t_n = (\gamma q \log n)^{1/\gamma}$, $r<q\le1$ fixed. By the tail behavior of $F$, we have
\beq 
\bar F(t_n) = L_n n^{-q}, \quad \bar F(t_n - \mu_n) = L_n n^{-(q^{1/\gamma}-r^{1/\gamma})^\gamma},
\eeq 
where $L_n$ denotes any factor logarithmic in $n$.

If $\gamma >1$, we define $r_\gamma = (1 - 2 ^{-1/(\gamma-1)})^\gamma$. If $r < r_\gamma$, we set $q=r/r_\gamma$. Then the LHS in \eqref{condHC1} is
\beq 
n(L_n n^{-r/r_\gamma} \vee \eps L_n n^{-r(r_\gamma^{-1/\gamma} - 1)^\gamma}) \asymp L_n ( n^{1-r/r_\gamma} \vee n^{1-\beta-r(r_\gamma^{-1/\gamma} - 1)^\gamma}) \gg \log^2 n.
\eeq
The LHS in \eqref{condHC2} is 
\beq 
L_n \frac{n^{\frac{1}{2}-\beta}(n^{-r(r_\gamma^{-1/\gamma} - 1)^\gamma} - n^{-r/r_\gamma})}{\sqrt{n^{-r/r_\gamma} + \eps \eta n^{-r(r_\gamma^{-1/\gamma} - 1)^\gamma}}} 
\asymp L_n (n^{\frac{1 + r/r_\gamma}{2}- \beta -r(r_\gamma^{-1/\gamma} - 1)^\gamma } \wedge n^{\frac{1}{2}(1- \beta-r(r_\gamma^{-1/\gamma} - 1)^\gamma )}),
\eeq  
where both exponents are positive when $r > (2^{1/(\gamma-1)}-1)^{\gamma-1}(\beta - \frac{1}{2})$.

If $r \ge r_\gamma$, we set $q=1$. Then the LHS in \eqref{condHC1} is
\beq 
n(L_n n^{-1} \vee \eps L_n n^{-(1-r^{1/\gamma})^\gamma}) 
\asymp L_n(1 \vee n^{1-\beta-(1-r^{1/\gamma})^\gamma}) \gg \log^2 n,
\eeq
if $1-\beta-(1-r^{1/\gamma})^\gamma > 0$. And the LHS in \eqref{condHC2} is 
\beq 
L_n \frac{n^{\frac{1}{2}-\beta}(n^{-(1-r^{1/\gamma})^\gamma} - n^{-1})}{\sqrt{n^{-1} + \eps \eta n^{-(1-r^{1/\gamma})^\gamma}}} \asymp L_n(n^{1-\beta -(1-r^{1/\gamma})^\gamma} \wedge n^{\frac{1}{2}(1-\beta -(1-r^{1/\gamma})^\gamma)}),
\eeq
where both exponents are positive when if $1-\beta-(1-r^{1/\gamma})^\gamma > 0$.

If $\gamma \le 1$, we set $q = r$, so that $t_n = \mu_n$. Then the LHS in \eqref{condHC1} is
\beq
n (L_n n ^{-r} \vee L_n) \gg \log^2 n,
\eeq 
and the LHS in \eqref{condHC2} is
\beq
L_n \frac{n^{\frac12 - \beta}(1 - n^{-r})}{\sqrt{n^{-r} + n^{-\beta}}} \asymp L_n n^{\frac12 - \beta + \frac{r}{2}},
\eeq 
where the exponent is positive when  $r > 2 \beta -1$. Comparing with the detection boundary, we see that the two-sample higher criticism test achieves the detection boundary in the generalized Gaussian model in the sparse regimes for any $\gamma > 0$.

In the dense regime, with parameterization \eqref{eps2sample} and \eqref{mu2}, $t = F^{-1}(1/2) = 0$, hence, 
\beq
\sqrt{n} \eps [\bar F(- \mu) - \frac{1}{2}] \asymp \sqrt{n} \eps \mu = n^{s-\beta},
\eeq
and the exponent is positive when $s<\beta$. So that the two-sample HC test achieves the detection boundary when $\gamma\ge 1/2$.

\section{Other tests}
It is well-known that in the more classical setting, the two-sample situation is closely related to the one-sample tests for symmetry. Essentially, the main two-sample tests have the same relatively efficiency between them as the corresponding one-sample tests. We analyzed some classical tests in this section.

\subsection{The Wilcoxon test}
The Wilcoxon test is the classical nonparametric test for location shift between two samples \cite{wilcoxon1945, mann1947test}. In particular, in this case, it rejects for large values of the Wilcoxon statistic $U$ which counts the number of pairs $X_i$, $Y_j$ with $X_i < Y_j$.

\begin{prp}
For the testing problem \eqref{problem2sample} and under \eqref{mn}, the Wilcoxon test is asymptotically powerful  (resp. powerless) when
\beq \label{wilcoxon}
\sqrt{n} \eps \big[ \frac{1}{2} -  \int F(\cdot - \mu) d F \big] \gg \log n \quad (\textit{reps.} \to 0).
\eeq 
\end{prp}

\begin{proof}
Mann and Whitney \cite{mann1947test} proved that under the null $F=G$, for large samples,
\beq 
\frac{\frac{U}{mn} - \E_0(\frac{U}{mn})}{\sigma_0(\frac{U}{mn})}
\eeq 
is approximately normally distributed. In particular, the first two moments of $U$ are \cite{mann1947test, lehmann1953power}
\beq 
\E \Big(\frac{U}{mn} \Big) = \int F d G,
\eeq 
\beq
mn\Var \Big(\frac{U}{mn}\Big) = \Big [\frac{m+n+1}{12} + (m-1)(\lambda - \eps_1) + (n-1)(\lambda - \eps_2) - \lambda^2 (m+n-1)\Big],
\eeq
where
\beq 
\lambda = \frac{1}{2} - \int F d G, \quad \eps_1 = \frac{1}{3} - \int F^2 dG, \quad \eps_2 = \frac{1}{3} - \int (1-G)^2 dF.
\eeq 
Hence, we have 
\beq 
\E_0 \Big(\frac{U}{mn} \Big) = \int F d F = \frac{1}{2},
\eeq 
\beq 
\Var_0 \Big(\frac{U}{mn} \Big) = \frac{m+n+1}{12mn}.
\eeq 
By Chebyshev's inequality, we have
\beq
\P_0 \Big(|U - \frac{mn}{2}| \ge a_n \sqrt{\frac{(m+n+1)mn}{12}} \Big) \to 0,
\eeq 
for any sequence $a_n$ diverging to infinity. Under $\cH_1$, $G = (1-\eps)F + \eps F(\cdot - \mu)$, we have
\beq 
\E_1 \Big(\frac{U}{mn} \Big) = \int F d G = \frac{1}{2} + \frac{\eps}{2} - \eps \int F(\cdot - \mu) d F,
\eeq 
and 
\beq 
\Var_1 \Big(\frac{U}{mn} \Big) = O \Big (\frac{m+n+1}{12mn} \Big),
\eeq 
as $0 \le \int F^k d G \le 1$, $k = 1,2$ and $0 \le \int (1-G)^2 dF \le 1$. Then by Chebyshev's inequality, we have
\beq
\P_1 \Big(|U - mn(\frac{1}{2} + \frac{\eps}{2} - \eps \int F(\cdot - \mu) d F)| \ge a_n \sqrt{\frac{(m+n+1)mn}{12}} \Big) \to 0,
\eeq 
for any sequence $a_n$ diverging to infinity. We choose $a_n = \log n$ and consider the test with rejection region $\{U - \frac{mn}{2} \ge \log n \sqrt{\frac{(m+n+1)mn}{12}}\}$.
The test is asymptotically powerful when, eventually, 
\beq
\eps \big [\frac{1}{2} -  \int F(\cdot - \mu) d F \big] \ge \log n \sqrt{\frac{m+n+1}{12mn}},
\eeq 
which is satisfied under \eqref{mn} and \eqref{wilcoxon}.

Next we show that the Wilcoxon test is asymptotically powerless when \eqref{wilcoxon} converges to zero. Lehmann \cite{lehmann1951consistency} showed that asymptotic normality still holds for $U$ under the alternative and \eqref{mn}, that is
\beq 
\frac{\frac{U}{mn} - \E_1(\frac{U}{mn})}{\sigma_1(\frac{U}{mn})} \to \cN(0,1).
\eeq 
Hence, under $\cH_1$, we have
\beq 
\frac{U - \E_0(U)}{\sigma_0(U)} = 
\Big( \frac{U - \E_1(U)}{\sigma_1(U)}  + \frac{\E_1(U) - \E_0(U)}{\sigma_1(U)}  \Big ) \cdot \frac{\sigma_1(U)}{\sigma_0(U)},
\eeq 
where $\frac{\E_1(U) - \E_0(U)}{\sigma_1(U)} \asymp \sqrt{n} \eps [ \frac{1}{2} -  \int F(\cdot - \mu) d F] \to 0$ and $\sigma_1(U)/\sigma_0(U) \asymp 1$. Therefore, by Slutsky's theorem, $(U - \E_0(U))/\sigma_0(U)$ also converges to $\cN(0,1)$ as under the null. No test based on $U$ would have any power.
\end{proof}
Note that the Wilcoxon test is asymptotically powerless when $\sqrt{n}\eps_n \to 0$. In the generalized Gaussian mixture model, in the dense regime with parameterization \eqref{eps2sample} and \eqref{mu2}, we have
\beq
\sqrt{n} \eps_n [ \frac{1}{2} -  \int F(\cdot - \mu_n) d F] \approx \sqrt{n} \eps_n [ \frac{1}{2} -  \int (F - \mu_n f) d F]  = 
\sqrt{n} \eps_n \mu_n \int f d F \asymp n^{s-\beta},
\eeq 
where $f$ is the density function. Hence, the Wilcoxon test is asymptotically powerful when $s > \beta$, and it achieves the detection boundary when $\gamma > 1/2$.

\subsection{The two-sample Kolmogorov-Smirnov test}
The two-sample Kolmogorov-Smirnov test \cite{smirnov1939estimation} rejects for larges values of 
\beq 
D_{m,n} = \sup_{t \in \mathbb{R}} [F_m(t) - G_n(t)].
\eeq 

\begin{prp}
	For the testing problem \eqref{problem2sample} and under \eqref{mn}, the two-sample Kolmogorov-Smirnov test is asymptotically powerful (resp. powerless) when 
	\beq \label{smirnov}
	\sqrt{n} \eps \sup_{t \in \mathbb{R}}  [\bar F(t-\mu) -\bar F(t)] \to \infty \quad (\textit{resp.} \to 0).
	\eeq 
\end{prp}

\begin{proof}
We already know the limiting distribution of $\sqrt{mn/(m+n)} D_{m,n}$  under the null hypothesis \cite{smirnov1939estimation}. Under $\cH_1$, by triangle inequality, 
\begin{align}
\sqrt{n} \sup_{t \in \mathbb{R}} [F_m(t) - G_n(t)] 
&\ge \sqrt{n} \sup_{t \in \mathbb{R}} [F(t) - G(t)] - \sqrt{n} \sup_{t \in \mathbb{R}} |F_m(t) - F(t)| - \sqrt{n} \sup_{t \in \mathbb{R}} |G_n(t) - G(t)|\\
& = \sqrt{n} \eps \sup_{t \in \mathbb{R}} [F(t) - F(t-\mu)] - O_p(1) \to \infty,
\end{align}
when the limit in \eqref{smirnov} is infinity.	

When the  limit in \eqref{smirnov} is 0, let $I_0$ and $I_1$ index the observations in the $Y$- sample coming from the null and contaminated components, respectively. Let $G_n^j(t) = \frac{1}{|I_j|} \sum_{i \in I_j}\IND{y_i \le t}$, $j = 0,1$. We have
\beq 
G_n(t) = \frac{|I_0|}{n} G_n^0(t) + \frac{|I_1|}{n} G_n^1(t).
\eeq 
By triangle inequality, 
\begin{align}
& |\sqrt{n} \sup_{t \in \mathbb{R}} [F_m(t) - G_n(t)] - \sqrt{|I_0|} \sup_{t \in \mathbb{R}} [F_m(t) - G_n^0(t)]|\\ 
& \le \big |\sqrt{\frac{|I_0|}{n}}-1 \big| \big|\sqrt{|I_0|} \sup_{t \in \mathbb{R}} [F_m(t) - G_n^0(t)]\big| + \sqrt{\frac{|I_1|}{n}} \big|\sqrt{|I_1|} \sup_{t \in \mathbb{R}} [F_m(t) - G_n^1(t)]\big|\\
& \le \big |\sqrt{\frac{|I_0|}{n}}-1 \big|  O_p(1) + \sqrt{\frac{|I_1|}{n}} \big|\sqrt{|I_1|}  \sup_{t \in \mathbb{R}} [F(t) - F(t-\mu)] + O_p(1) \big| = o_p(1),
\end{align}
by the fact that $|I_0| \sim_p n$, $|I_1| \sim_p n\eps$ and \eqref{smirnov}  converges to 0. Hence, $\sqrt{n} D_{m,n} \sim \sqrt{|I_0|} D_{m,|I_0|}$ under $\cH_1$, which has the same limiting distribution as under $\cH_0$.
\end{proof}

Note that the two-sample Kolmogorov-Smirnov test has no power in the sparse regime. In the generalized Gaussian mixture model, in the dense regime with parameterization \eqref{eps2sample} and \eqref{mu2}, we have
\beq
\sqrt{n} \eps \sup_{t \in \mathbb{R}}  [\bar F(t-\mu) -\bar F(t)] \ge \sqrt{n} \eps  [\bar F(-\mu) -\bar F(0)] \asymp n^{s-\beta} \to \infty,
\eeq
when $s > \beta$. Same as the Wilcoxon test, it only achieves the detection boundary with $\gamma > 1/2$.

\subsection{The tail-run test}
We now consider the tail-run test. Let $\zeta_{(j)} = 0$ or 1, according to whether the $j$th largest observation is from the $X$-sample or the $Y$-sample, $j = 1, \cdots, m+n$. Then the tail-run test rejects for large values of 
\beq \label{tail}
L ^ * = \max \{l \ge 0: \zeta_{(1)} = \cdots = \zeta_{(l)} = 1 \}.
\eeq
The one-sample tail-run test for sparse mixtures is investigated in \cite{AriasCastro:2016is}. It is also analogous to the extreme tests in the normal mixture model.

\begin{prp}
	For the testing problem \eqref{problem2sample} and under \eqref{mn}, and let $(l_n)$ be a divergent sequence of positive integers. The tail-run test is asymptotically powerful when there exits a sequence $(t_n)$ such that 
	\beq \label{tailcond}
	m\bar F(t_n) \to 0, \quad n\eps \bar F(t_n -\mu) \ge 2 l_n.
	\eeq 
\end{prp}

\begin{proof}
	We consider the tail-run test with rejection region $\{L^* \ge l_n\}$. Note that $L^*$ is the number of the $Y$ - samples until the first $X$-sample is encountered. Under $\cH_0$, $L^*$ is following negative hypergeometric distribution with the population size $m+n$, and we have
	\beq 
	\E_0(L^*) = \frac{n}{m+1}, \quad \Var_0(L^*)  = \frac{(m+n+1)n}{(m+1)(m+2)} [1 - \frac{1}{m+1}].
	\eeq 
	Hence, $L^* = O_p(1)$ and $l_n \to \infty$, we have $\P_0(L^* \ge l_n) \to 0$ as $n \to \infty$.
	
	Under $\cH_1$, note that 
	\beq
	\P_X(\max_i X_i \le t_n) = (1 - \bar F(t_n))^m \to 1, 
	\eeq 
	under the condition $m\bar F(x_n) \to 0$. Therefore, $L^* \ge N := \#\{j, Y_j > t_n\}$ with high probability. And $N \sim \Bin(n,p_y)$ where $p_y = (1-\eps)\bar F(t_n) + \eps \bar F(t_n -\mu)$. Eventually, under \eqref{tailcond} , we have $N = (1+o_p(1)) n p_y \ge l_n$.
\end{proof}
In the generalized Gaussian mixture model, with parameterization \eqref{eps2sample} and \eqref{mu}, we choose $t_n = (\gamma(1+q) \log n)^{1/\gamma}$, $q>0$ fixed. By the tail behavior of $F$, we have
\beq 
m \bar F(t_n) = L_n n^{-q}, \quad n \eps \bar F(t_n - \mu_n) = L_n n^{1-\beta-((1+q)^{1/\gamma}-r^{1/\gamma})^\gamma},
\eeq 
where $L_n$ denotes any factor logarithmic in $n$. When $r > (1 - (1-\beta)^{1/\gamma})^\gamma$ is fixed, we can choose $q>0$ small enough that $1-\beta-((1+q)^{1/\gamma}-r^{1/\gamma})^\gamma> 0$. Hence, the tail-run test is suboptimal in the moderately sparse regime and is optimal in the very sparse regime. 

\section{Numerical experiments}
We performed some numerical experiments to investigate the finite sample performance of the likelihood ratio test (LRT), the two-sample higher criticism (HC) test, the Wilcoxon test, the two-sample Kolmogorov-Smirnov (KS) test and the tail-run test.  We set sample sizes $m = n = 10^5$ in order to capture the large-sample behavior of these tests. The p-values for each test are calibrated as follows:

\renewcommand{\theenumi}{(\alph{enumi})}
\renewcommand{\labelenumi}{\theenumi}
\benum  \setlength{\itemsep}{0in}
\item For the {\em likelihood ratio test} and the {\em two-sample higher criticism test}, we simulated the null distribution based on $4,000$ Monte Carlo replicates. 
\item For the {\em Wilcoxon test} and the {\em two-sample Kolmogorov-Smirnov test}, the p-values are from the limiting distributions.
\item For the {\em tail-run test}, we used the exact null distribution, that is the negative hypergeometric distribution.
\eenum

For each scenario, we repeated the whole process 200 times and recorded the fraction of p-values  smaller than 0.05, representing the empirical power at the 0.05 level. 

\subsection*{Normal mixture model}
In this model,  $F$ is standard normal.  The results are reported in \figref{numerics41} and are largely congruent with the theory developed earlier.  

{\em Dense regime.} We set $\beta = 0.2$ and $\mu_n = n^{s-1/2}$ with $s$ ranging from 0.05 to 0.5 with increments of 0.05. The two-sample HC test, the Wilcoxon test and the two-sample KS test perform comparable to the LRT, while the tail-run test is obviously suboptimal.

{\em Moderately sparse regime.} We set $\beta = 0.6$ and $\mu_n = \sqrt{2r\log n}$ with $r$ ranging from 0.05 to 0.5 with increments of 0.05. The two-sample HC performs slightly worse than the LRT but better than the tail-run test, while the Wilcoxon test and the KS test are powerless. 

{\em Very sparse regime.} We set $\beta = 0.8$ and $\mu_n = \sqrt{2r\log n}$ with $r$ ranging from 0.1 to 0.9 with increments of 0.1. Though our theory show that both the two-sample HC test and the tail-run test achieve the detection boundary, they both perform significantly below the LRT. The tail-run test is more powerful than the two-sample HC test, which is consistent with the observation in the one-sample setting \cite{AriasCastro:2018wr}. 

\begin{figure}[ht!]
	\centering	
	\subfigure[$\beta=0.2$]{
		\includegraphics[width=0.45\textwidth]{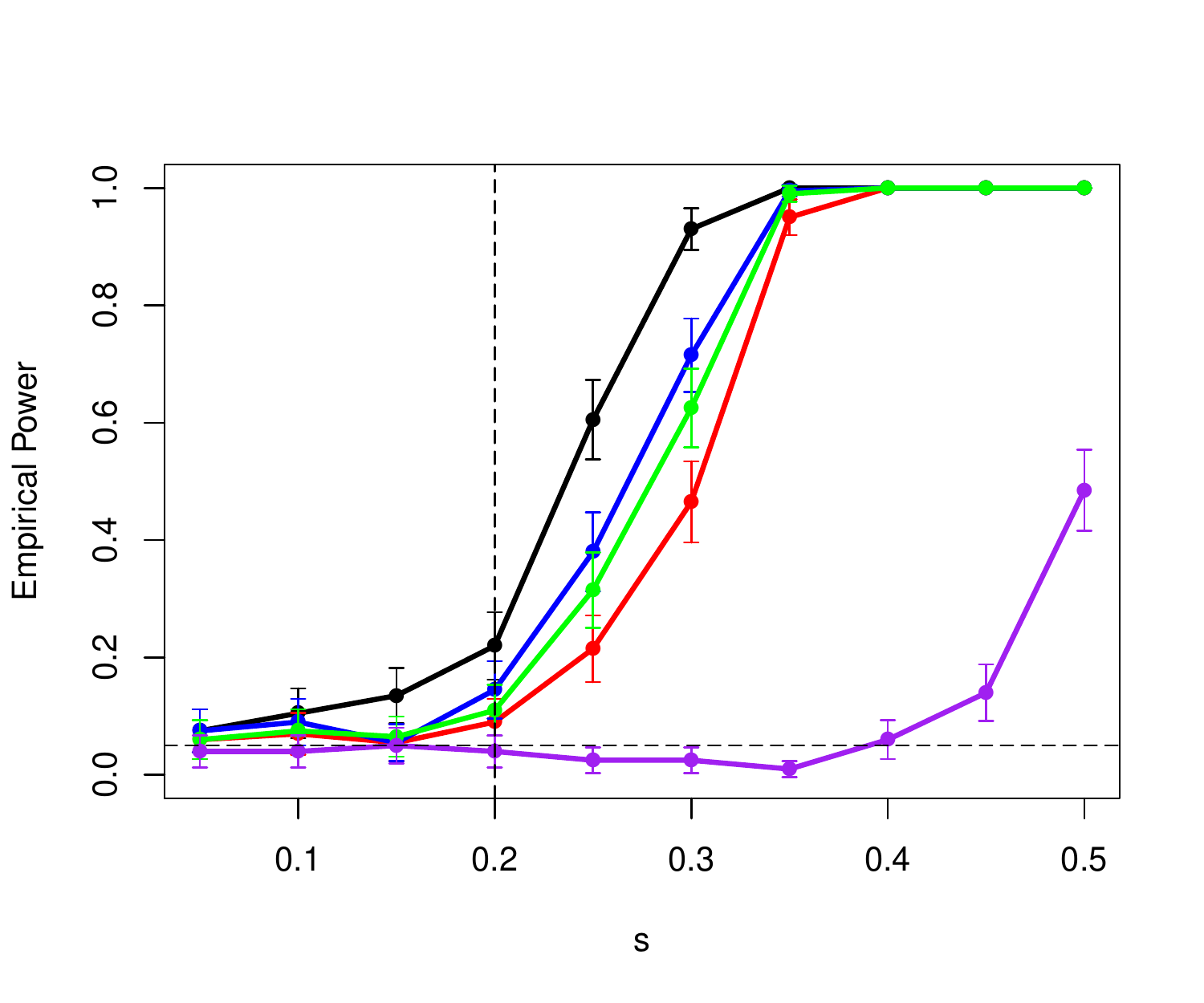}
		\label{fig:subfig41}
	}	
	\subfigure[$\beta=0.6$]{
		\includegraphics[width=0.45\textwidth]{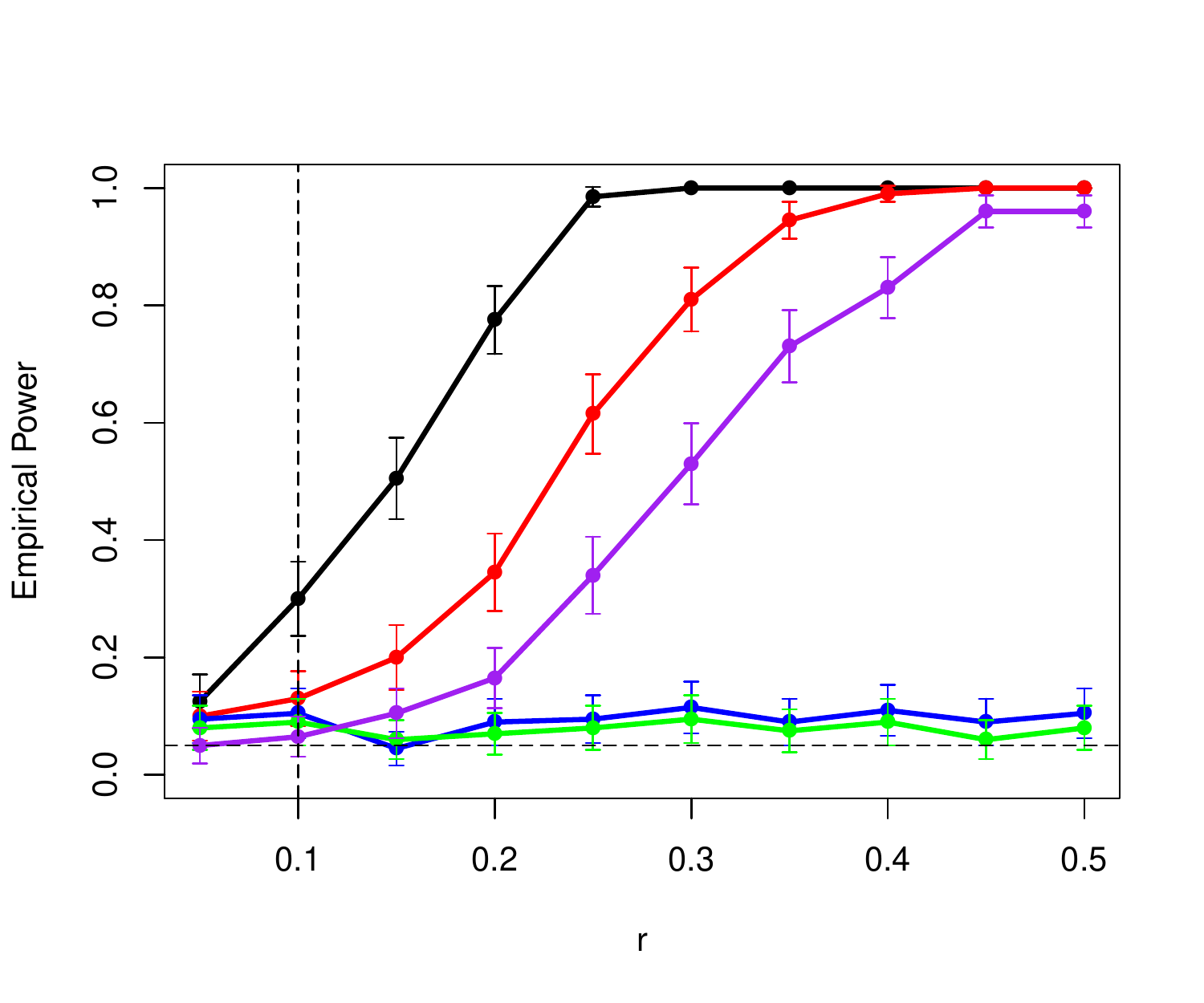}
		\label{fig:subfig42}
	}
	\subfigure[$\beta=0.8$]{
		\includegraphics[width=0.45\textwidth]{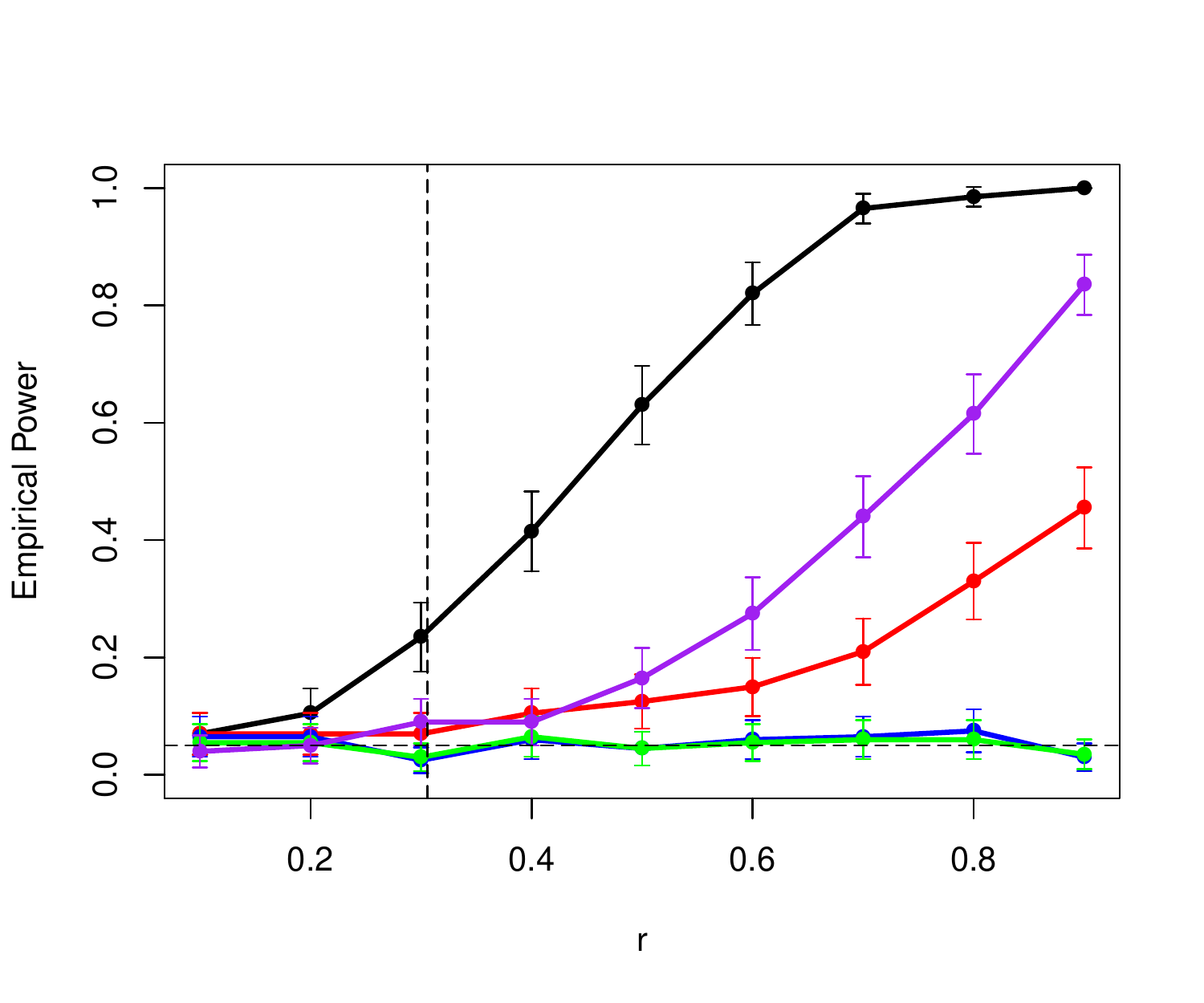}
		\label{fig:subfig43}
	}
	\caption[Empirical power comparison with 95\% error bars in the normal mixture model.] {Empirical power comparison with 95\% error bars for the likelihood ratio test (black), the two-sample higher criticism test (red), the Wilcoxon test (blue), the two-sample Kolmogorov-Smirnov test (green) and the tail-run test (purple). \subref{fig:subfig41}  Dense regime where $\beta = 0.2$. \subref{fig:subfig42} Moderately sparse regime where $\beta = 0.6$. \subref{fig:subfig43} Very sparse regime where $\beta = 0.8$.  The horizontal line marks the level (set at 0.05) and the vertical line marks the asymptotic detection boundary derived earlier. The sample size is $m = n =10^5$ and the power curves and error bars are based on 200 replications.}
	\label{fig:numerics41}
\end{figure}

\subsection*{Double-exponential mixture model}
In this model, $F$ is double-exponential with variance 1. The simulation results are reported in \figref{numerics42}. The results are largely congruent with our theory.

\begin{figure}[ht!]
	\centering	
	\subfigure[$\beta=0.2$]{
		\includegraphics[width=0.45\textwidth]{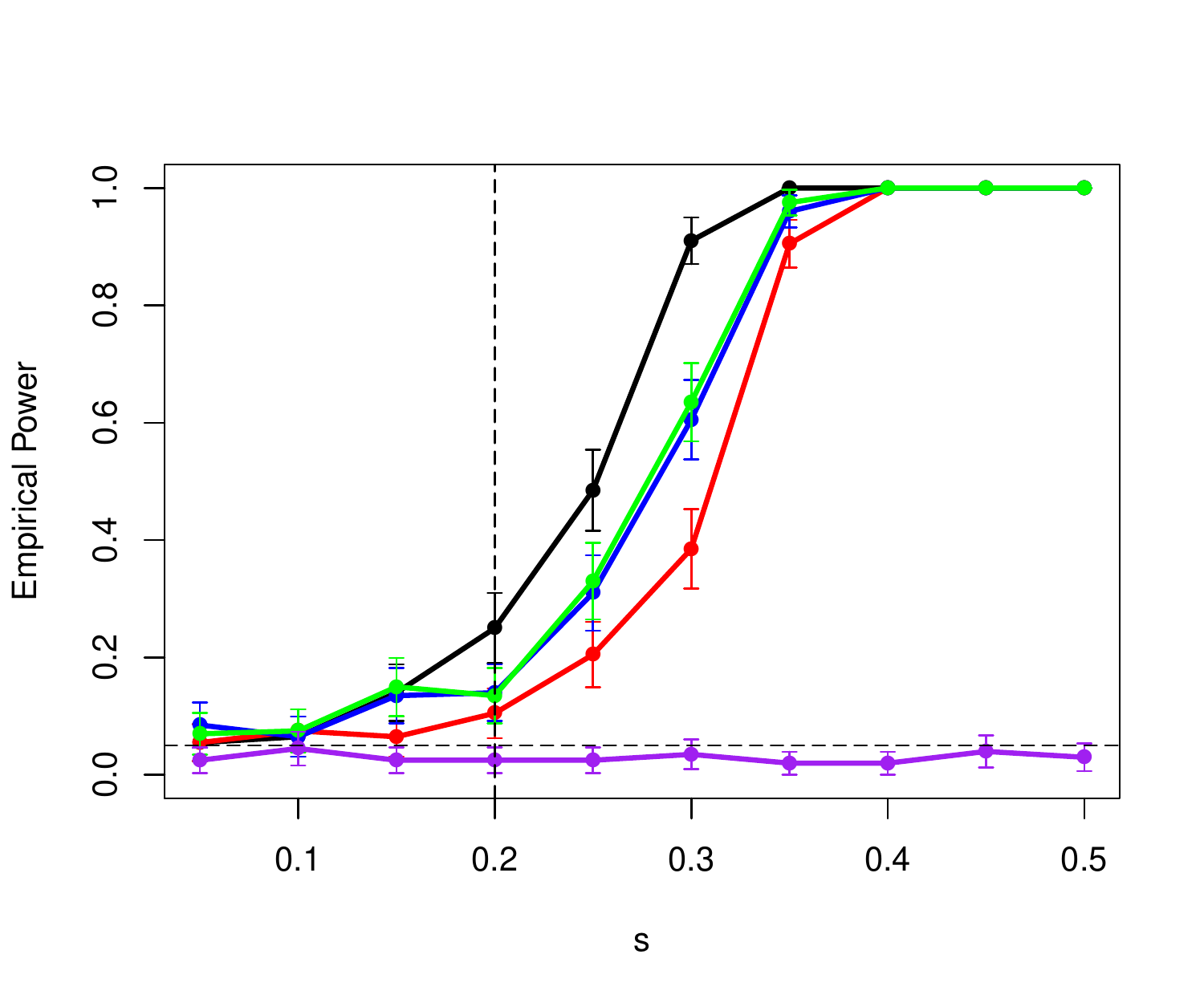}
		\label{fig:subfig45}
	}	
	\subfigure[$\beta=0.6$]{
		\includegraphics[width=0.45\textwidth]{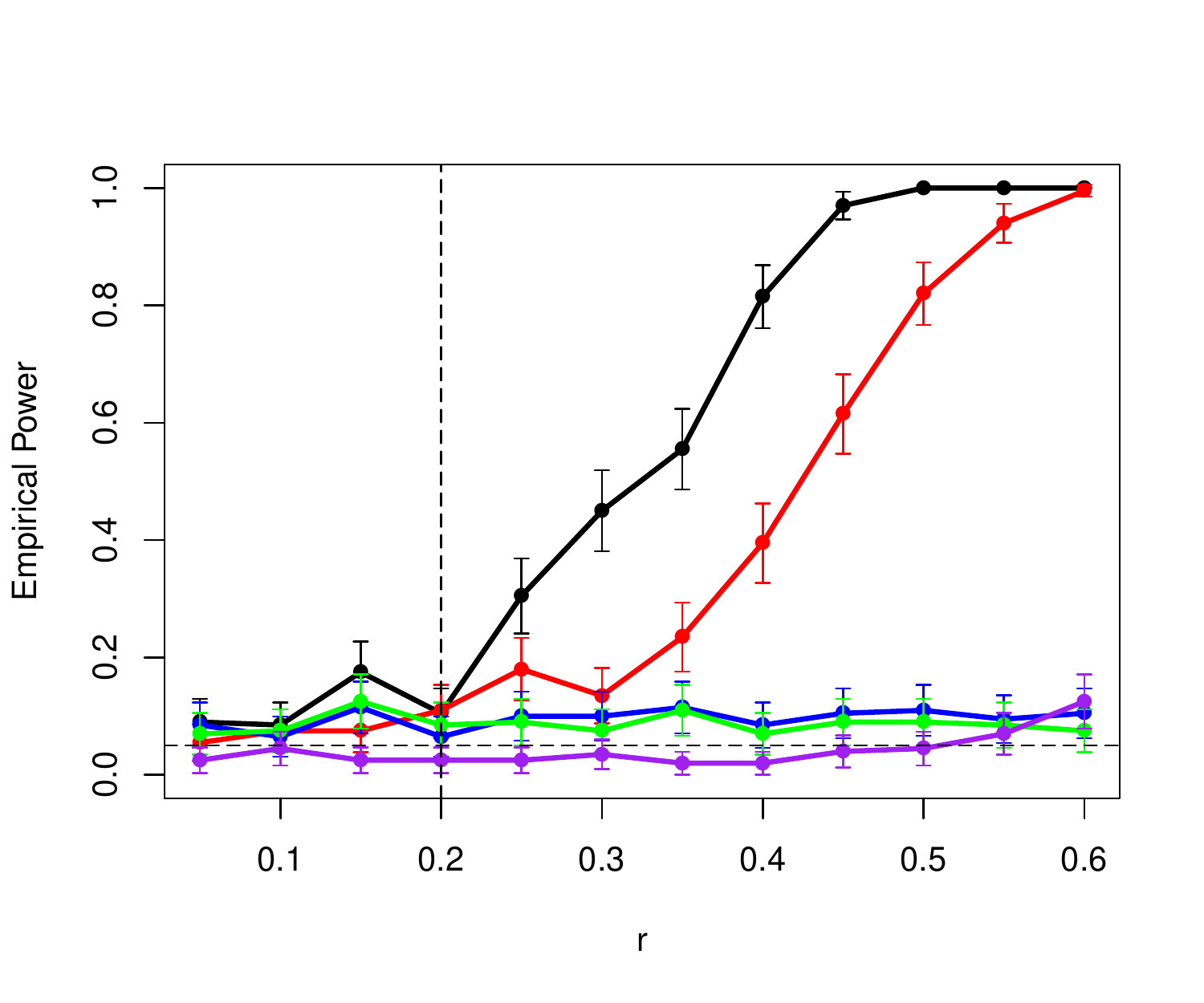}
		\label{fig:subfig46}
	}
	\caption[Empirical power comparison with 95\% error bars in the double-exponential model.] {Empirical power comparison with 95\% error bars for the likelihood ratio test (black), the two-sample higher criticism test (red), the Wilcoxon test (blue), the two-sample Kolmogorov-Smirnov test (green) and the tail-run test (purple). \subref{fig:subfig41}  Dense regime where $\beta = 0.2$. \subref{fig:subfig42} Moderately sparse regime where $\beta = 0.6$. \subref{fig:subfig43} Very sparse regime where $\beta = 0.8$.  The horizontal line marks the level (set at 0.05) and the vertical line marks the asymptotic detection boundary derived earlier. The sample size is $m = n =10^5$ and the power curves and error bars are based on 200 replications.}
	\label{fig:numerics42}
\end{figure}

\bibliographystyle{abbrvnat}
\bibliography{ref-variance}

\begin{thebibliography}{17}
\providecommand{\natexlab}[1]{#1}
\providecommand{\url}[1]{\texttt{#1}}
\expandafter\ifx\csname urlstyle\endcsname\relax
  \providecommand{\doi}[1]{doi: #1}\else
  \providecommand{\doi}{doi: \begingroup \urlstyle{rm}\Url}\fi

\bibitem[Arias-Castro and Huang(2020)]{AriasCastro:2018wr}
E.~Arias-Castro and R.~Huang.
\newblock The sparse variance contamination model.
\newblock \emph{Statistics}, 54\penalty0 (5):\penalty0 1081--1093, 2020.
\newblock \doi{10.1080/02331888.2020.1823394}.

\bibitem[Arias-Castro and Wang(2016)]{AriasCastro:2016is}
E.~Arias-Castro and M.~Wang.
\newblock {Distribution-free tests for sparse heterogeneous mixtures}.
\newblock \emph{TEST}, 26\penalty0 (1):\penalty0 71--94, 2016.

\bibitem[Arias-Castro and Wang(2018)]{arias2018dist}
E.~Arias-Castro and M.~Wang.
\newblock Distribution-free tests for sparse heterogeneous mixtures.
\newblock \emph{arXiv preprint arXiv:1308.0346}, 2018.

\bibitem[Cai and Wu(2014)]{cai2014optimal}
T.~T. Cai and Y.~Wu.
\newblock Optimal detection of sparse mixtures against a given null
  distribution.
\newblock \emph{IEEE Transactions on Information Theory}, 60\penalty0
  (4):\penalty0 2217--2232, 2014.

\bibitem[Canner(1975)]{canner1975simulation}
P.~L. Canner.
\newblock A simulation study of one-and two-sample kolmogorov-smirnov
  statistics with a particular weight function.
\newblock \emph{Journal of the American Statistical Association}, 70\penalty0
  (349):\penalty0 209--211, 1975.

\bibitem[Conover and Salsburg(1988)]{conover1988locally}
W.~J. Conover and D.~S. Salsburg.
\newblock Locally most powerful tests for detecting treatment effects when only
  a subset of patients can be expected to ``respond" to treatment.
\newblock \emph{Biometrics}, 44:\penalty0 189--196, 1988.

\bibitem[Donoho and Jin(2004)]{Jin:2004fj}
D.~Donoho and J.~Jin.
\newblock {Higher criticism for detecting sparse heterogeneous mixtures}.
\newblock \emph{The Annals of Statistics}, 32\penalty0 (3):\penalty0 962--994,
  2004.

\bibitem[Finner and Gontscharuk(2018)]{finner2018two}
H.~Finner and V.~Gontscharuk.
\newblock {T}wo-sample {K}olmogorov-{S}mirnov-type tests revisited: old and new
  tests in terms of local levels.
\newblock \emph{The Annals of Statistics}, 46\penalty0 (6A):\penalty0
  3014--3037, 2018.

\bibitem[Ingster(1997)]{ingster1997some}
Y.~I. Ingster.
\newblock Some problems of hypothesis testing leading to infinitely divisible
  distributions.
\newblock \emph{Mathematical Methods of Statistics}, 6\penalty0 (1):\penalty0
  47--69, 1997.

\bibitem[Jaeschke(1979)]{jaeschke1979asymptotic}
D.~Jaeschke.
\newblock The asymptotic distribution of the supremum of the standardized
  empirical distribution function on subintervals.
\newblock \emph{The Annals of Statistics}, 7\penalty0 (1):\penalty0 108--115,
  1979.

\bibitem[Lehmann(1951)]{lehmann1951consistency}
E.~L. Lehmann.
\newblock Consistency and unbiasedness of certain nonparametric tests.
\newblock \emph{The Annals of Mathematical Statistics}, 22:\penalty0 165--179,
  1951.

\bibitem[Lehmann(1953)]{lehmann1953power}
E.~L. Lehmann.
\newblock The power of rank tests.
\newblock \emph{The Annals of Mathematical Statistics}, 24\penalty0
  (1):\penalty0 23--43, 1953.

\bibitem[Mann and Whitney(1947)]{mann1947test}
H.~B. Mann and D.~R. Whitney.
\newblock On a test of whether one of two random variables is stochastically
  larger than the other.
\newblock \emph{The Annals of Mathematical Statistics}, 18:\penalty0 50--60,
  1947.

\bibitem[Pettitt(1976)]{pettitt1976two}
A.~N. Pettitt.
\newblock A two-sample {A}nderson-{D}arling rank statistic.
\newblock \emph{Biometrika}, 63\penalty0 (1):\penalty0 161--168, 1976.

\bibitem[Smirnov(1939)]{smirnov1939estimation}
N.~V. Smirnov.
\newblock On the estimation of the discrepancy between empirical curves of
  distribution for two independent samples.
\newblock \emph{Bull. Mathematics University Moscow}, 2:\penalty0 3--16, 1939.

\bibitem[Wilcoxon(1945)]{wilcoxon1945}
F.~Wilcoxon.
\newblock Individual comparisons by ranking methods.
\newblock \emph{Biometrics}, 1:\penalty0 80--83, 1945.

\bibitem[Zhao et~al.(2017)Zhao, Cai, and Li]{Zhao:9999ku}
S.~D. Zhao, T.~T. Cai, and H.~Li.
\newblock {Optimal detection of weak positive latent dependence between two
  sequences of multiple tests}.
\newblock \emph{Journal of Multivariate Analysis}, 160:\penalty0 169--184,
  2017.

\end{thebibliography}

\end{document}